\documentclass[11pt]{amsart}
\usepackage{amsmath,amsthm,amsfonts,amscd,amssymb,eucal,latexsym,mathrsfs}

\newtheorem{theorem}{Theorem}[section]

\newtheorem{lemma}[theorem]{Lemma}

\theoremstyle{definition}

\newcommand{\cQ}{{\mathcal Q}}
\newcommand{\cP}{{\mathcal P}}
\newcommand{\cR}{{\mathcal R}}

\newcommand{\Cb}{{\mathbb C}}

\newcommand{\Zb}{{\mathbb Z}}

\newcommand{\Nb}{{\mathbb N}}

\newcommand{\Sym}{{\rm Sym}}
\newcommand{\Hom}{{\rm Hom}}

\allowdisplaybreaks

\begin{document}

\title[Bernoulli actions and infinite entropy]{Bernoulli actions and infinite entropy}

\author{David Kerr}
\author{Hanfeng Li}
\address{\hskip-\parindent
David Kerr, Department of Mathematics, Texas A{\&}M University,
College Station TX 77843-3368, U.S.A.}
\email{kerr@math.tamu.edu}

\address{\hskip-\parindent
Hanfeng Li, Department of Mathematics, SUNY at Buffalo,
Buffalo NY 14260-2900, U.S.A.}
\email{hfli@math.buffalo.edu}

\date{May 27, 2010}

\begin{abstract}
We show that, for countable sofic groups, a Bernoulli action with infinite entropy base has infinite entropy
with respect to every sofic approximation sequence. This builds on the work of Lewis Bowen in the case of finite
entropy base and completes the computation of measure entropy for Bernoulli actions over countable sofic groups.
One consequence is that such a Bernoulli action fails to have a generating countable partition
with finite entropy if the base has infinite entropy, which in the amenable case is well known 
and in the case that the acting group contains the
free group on two generators was established by Bowen using a different argument.
\end{abstract}

\maketitle

\section{Introduction}

In \cite{Bow10} Lewis Bowen introduced a notion of entropy for measure-preserving actions of countable
sofic groups admitting a generating partition with finite entropy. This measure entropy is defined relative
to a given sofic approximation sequence for the group and thus yields a collection of numerical invariants in general.
For a Bernoulli action with finite base entropy, Bowen showed that the sofic measure entropy
is equal to the base entropy for every choice of sofic approximation sequence. As a consequence
he was able to extend the entropy classification of Bernoulli actions in the amenable setting, 
due to Ornstein for $\Zb$ and to Ornstein and Weiss more generally,
to the class of all countable sofic groups having the property that any two Bernoulli actions
with the same base entropy are conjugate, which includes all countable sofic groups
containing an element of infinite order.
Moreover, Bowen demonstrated that Bernoulli actions with nontrivial bases over a given countable sofic
group containing the free group $F_2$ are all weakly isomorphic \cite{Bow10b}, which enabled him to
conclude in \cite{Bow10} that Bernoulli actions with infinite entropy base over a countable sofic group
containing $F_2$ do not admit a generating partition with finite entropy, thereby answering a question of Weiss.

In \cite{KerLi10T} the present authors extended Lewis Bowen's sofic measure entropy to general
measure-preserving actions of a countable sofic group on a standard probability space by recasting
the definition in operator-algebraic terms, much in the spirit of Rufus Bowen's approach to topological
entropy for $\Zb$-actions, which replaces the analysis of set intersections with the counting of
$\varepsilon$-separated partial orbits. The main goal of \cite{KerLi10T} was to introduce a topological
version of sofic measure entropy and establish a variational principle relating the two.
In this note we respond to a question that Bowen posed to the authors 
by showing that, for countable sofic groups, a Bernoulli action with infinite entropy base has 
infinite entropy with respect to every sofic approximation sequence.
The argument makes use of Bowen's finite entropy lower bound and
is carried out by representing the dynamics in a topological way.
It follows that the entropy of any Bernoulli action
of a countable sofic group over a standard base is equal to the base entropy, independently
of the sofic approximation sequence (Theorem~\ref{T-entropy of Bernoulli}). As a consequence,
conjugate Bernoulli actions of a countable sofic group over standard bases have the same
base entropy (Theorem~\ref{T-isomorphism of Bernoulli}), a fact that Bowen proved under the assumption
that the bases have finite entropy or that the group is Ornstein, i.e., has the property that
any two Bernoulli actions with the same base entropy are conjugate. Another consequence is that
a Bernoulli action of a countable sofic group with an infinite entropy standard base does not
admit a generating measurable partition with finite entropy (Theorem~\ref{T-countable generator}),
which is well known for countable amenable groups and, as mentioned above, was established by Bowen for
countable sofic groups containing $F_2$ using a different argument based on the weak isomorphism of
Bernoulli actions. Note that there exist countable sofic groups which are not
amenable and do not contain $F_2$. Indeed Ershov showed the existence of a countable nonamenable 
residually finite torsion group \cite[Cor.\ 8.5]{Ers08} (see also \cite{Osi09}).

We now recall some terminology and notation pertaining to sofic measure entropy. 
We refer the reader to \cite{KerLi10T} for more details. Let $G$ be a countable sofic group.
Let $\Sigma = \{ \sigma_i : G\to\Sym (d_i )\}_{i=1}^\infty$
be a sofic approximation sequence for $G$, i.e., $\{ d_i \}_{i=1}^\infty$ is a sequence of
positive integers satisfying $\lim_{i\to\infty} d_i = \infty$ and the maps $\sigma_i$ into the
permutation groups $\Sym (d_i )$ are asymptotically multiplicative and free in the sense that
\[
\lim_{i\to\infty}
\frac{1}{d_i} \big| \{ k\in \{ 1,\dots ,d_i \} : \sigma_{i,st} (k) = \sigma_{i,s} \sigma_{i,t} (k) \} \big| = 1
\]
for all $s,t\in G$ and
\[
\lim_{i\to\infty}
\frac{1}{d_i} \big| \{ k\in \{ 1,\dots ,d_i \} : \sigma_{i,s} (k) \neq \sigma_{i,t} (k) \} \big| = 1
\]
for all distinct $s,t\in G$.
The sofic measure entropy $h_{\Sigma ,\mu} (X,G)$ of a measure-preserving action of $G$
on a standard probability space $(X,\mu )$
with respect to $\Sigma$ is defined, roughly speaking, 
by measuring the exponential growth as $i\to\infty$ of the maximal cardinality of
$\varepsilon$-separated sets of approximately equivariant, approximately multiplicative, and
approximately measure-preserving maps from $L^\infty (X,\mu )$ into $\Cb^{d_i}$, where the latter
is equipped with the uniform probability measure. Instead of recalling the precise definition, which can be found as
Definition~2.2 in \cite{KerLi10T}, we will give here an equivalent formulation more suited to
our purpose, which requires us to endow our measure-theoretic framework with topological structure
in order to facilitate certain approximations. We thus suppose that
$X$ is a compact metrizable space, $\mu$ is a Borel probability measure on $X$, 
and $\alpha$ is a continuous measure-preserving action of $G$ on $X$.
The notation $\alpha$ will actually be reserved for the induced
action on $C(X)$, so that $\alpha_g (f)$
for $f\in C(X)$ and $g\in G$ will mean the function $x\mapsto f(g^{-1} x)$, 
with concatenation being used for the action on $X$.
Let $\cP$ be a finite partition of unity in $C(X)$ and let $d\in\Nb$.
On the set of unital homomorphisms from $C(X)$ to $\Cb^d$
we define the pseudometric
\begin{align*}
\rho_{\cP ,\infty} (\varphi , \psi ) &= \max_{p\in\cP} \| \varphi (p) - \psi (p) \|_\infty .
\end{align*}
Let $\sigma$ be a map from $G$ to the permutation group $\Sym (d)$ of the set $\{ 1,\dots ,d \}$.
We also use $\sigma$ to denote
the induced action on $\Cb^d \cong C(\{ 1,\dots ,d \} )$, i.e., for $f\in\Cb^d$ and $s\in G$ we write
$\sigma_s (f)$ to mean $f\circ\sigma_s^{-1}$.
We write $\zeta$ for the uniform probability measure on $\{ 1,\dots ,d \}$, and
$\| \cdot \|_2$ for the norm $f \mapsto \zeta (|f|^2 )^{1/2}$ on $\Cb^d$.
Let $F$ be a nonempty finite subset of $G$, $m\in\Nb$ and $\delta > 0$.
Define $\Hom_\mu^X (\cP , F,m,\delta ,\sigma )$ to be the set of all unital homomorphisms
$\varphi : C(X) \to \Cb^d$ such that
\begin{enumerate}
\item[(i)]
$| \zeta\circ\varphi (f) - \mu (f) | < \delta$ for all $f\in\cP_{F,m}$,

\item[(ii)]
$\| \varphi\circ\alpha_s (f) - \sigma_s \circ\varphi (f) \|_2 < \delta$ for all $f\in\cP$ and $s\in F$,
\end{enumerate}
where $\cP_{F,m}$ denotes the set of all products of the form
$\alpha_{s_1} (p_1 )\cdots \alpha_{s_j} (p_j )$ where $1\leq j\leq m$,
$p_1 , \dots p_j \in\cP$, and $s_1 , \dots ,s_j \in F$.
The measure-preserving version of Proposition~4.11 in \cite{KerLi10T} (which can be established
using the same argument) and the discussion in Section~5 of \cite{KerLi10T} together show that if $\cP$ is dynamically generating 
in the sense that there is no proper $G$-invariant unital $C^*$-subalgebra of $C(X)$ containing $\cP$, then 
\[
h_{\Sigma, \mu}(X,G)=\sup_{\varepsilon>0}\inf_F\inf_{m\in \Nb}\inf_{\delta>0}
\limsup_{i\to\infty} \frac{1}{d_i} \log N_{\varepsilon} (\Hom_{\mu}^{X} (\cP ,F,m,\delta ,\sigma_i ),\rho_{\cP,\infty} )
\]
where $N_\varepsilon (\cdot ,\rho_{\cP,\infty} )$
denotes the maximal cardinality of an $\varepsilon$-separated subset with respect to the pseudometric $\rho_{\cP,\infty}$
and $F$ ranges over all nonempty finite subsets of $G$.
\medskip

\noindent{\it Acknowledgements.}
The first author was partially supported by NSF grant DMS-0900938,
and the second author partially supported by NSF grant DMS-0701414.
We thank Lewis Bowen for pointing out the references \cite{Ers08,Osi09}.

\section{Results}\label{S-main}

Let $G$ be a countable sofic group and $(X,\mu )$ a standard probability space. 
Taking the product Borel structure on $X^G$ and the product measure $\mu^G$, we obtain a standard probability 
space $(X^G ,\mu^G )$ on which $G$ acts by the shifts $g\cdot (x_h )_{h\in G} = (x_{g^{-1} h} )_{h\in G}$, 
and we refer to this as a Bernoulli action.

The following lemma is a direct consequence of (and is easily seen to be equivalent to)
the lower bound for the sofic entropy of a Bernoulli action with finite entropy base that is part of
Theorem~8.1 in \cite{Bow10}. Recall that the entropy of a measurable partition $\cQ$ of a probability space 
$(X,\mu )$ is defined by
\[
H_\mu (\cQ) = - \sum_{Q\in\cQ} \mu (Q) \log \mu (Q) .
\]

\begin{lemma} \label{L-finite lower bound}
Let $G$ be a countable sofic group.
Let $\Sigma = \{ \sigma_i : G\to\Sym (d_i ) \}_{i=1}^\infty$ be a sofic approximation sequence for $G$.
Let $(X, \mu)$ be a standard probability space. Let $\cR$ be a finite measurable partition of
$X$. Then for every nonempty finite set $F\subseteq G$ and $\delta>0$ one has
\begin{align*}
\lefteqn{\limsup_{i\to \infty}\frac{1}{d_i}\log \bigg|\bigg\{\beta\in \cR^{\{1, \dots, d_i\}}:} \hspace*{10mm} \\
&\hspace*{10mm} \sum_{\varphi \in \cR^F}\bigg|\prod_{g\in F}\mu(\varphi(g))
-\zeta\bigg(\bigcap_{g\in F}\sigma_i(g)\beta^{-1}(\varphi(g))\bigg)\bigg|\le \delta\bigg\}\bigg|
\ge H_{\mu}(\cR).
\end{align*}
\end{lemma}

\begin{lemma} \label{L-lower bound}
Let $G$ be a countable sofic group. Let $\Sigma = \{ \sigma_i : G\to\Sym (d_i ) \}_{i=1}^\infty$ be a sofic approximation sequence for $G$.
Let $(X, \mu)$ be a standard probability space. Let $\cQ$ be a finite measurable partition of
$X$. Then $h_{\Sigma, \mu^G}(X^G, G)\ge H_{\mu}(\cQ)$.
\end{lemma}
\begin{proof}
We can identify $X$ with a Borel subset of $[0, 1]$ such that the closures of the atoms of $\cQ$ are pairwise disjoint.
Then $\mu$ extends to a Borel probability measure $\nu$ on $\bar{X}$ such that $\mu(A)=\nu(A)$ for every measurable
subset $A$ of $X$. It follows that the Bernoulli actions $(G, X^G, \mu^G)$ and $(G, \overline{X}^G, \nu^G)$ are measurably
isomorphic. Thus $h_{\Sigma, \mu^G}(X^G, G)=h_{\Sigma, \nu^G}(\bar{X}^G, G)$.
Denote by $\overline{\cQ}$ the partition of $\overline{X}$ consisting of $\overline{Q}$ for $Q\in \cQ$.
Then $H_{\mu}(\cQ)=H_{\nu}(\bar{\cQ})$.
Thus we may replace $X$ and $\cQ$ by $\overline{X}$ and $\overline{\cQ}$ respectively.
Therefore we will assume that $X$ is a closed subset of $[0, 1]$ and $\cQ$ is a closed and open partition of $X$.
Equipped with the product topology, $X^G$ is a compact metrizable space.
The shift action of $G$ on $X^G$ is continuous, and $\mu^G$ is a Borel probability measure on $X^G$.
We write $\alpha$ for this action as applied to $C(X^G )$, following the notational convention 
from the introduction.

Denote by $p$ the coordinate function on $X$, i.e. $p(x)=x$ for $x\in X$. Then $\cP:=\{p, 1-p\}$ is a partition of unity in $C(X)$ generating $C(X)$ as a unital $C^*$-algebra. Denote by $e$ the identity element of $G$. Via the coordinate map $X^G\rightarrow X$ sending $(x_g)_{g\in G}$ to $x_e$, we will also think of $\cP$ as a partition of unity in $C(X^G)$. Then $\cP$ dynamically generates $C(X^G)$,
and so according to the introduction we have
\[
h_{\Sigma, \mu^G}(X^G, G)=\sup_{\varepsilon>0}\inf_F\inf_{m\in \Nb}\inf_{\delta>0}\limsup_{i\to\infty} \frac{1}{d_i}
\log N_{\varepsilon} (\Hom_{\mu^G}^{X^G} (\cP ,F,m,\delta ,\sigma_i ),\rho_{\cP , \infty} ),
\]
where $F$ ranges over all nonempty finite subsets of $G$ and the pseudometric $\rho_{\cP ,\infty}$
on the set of unital homomorphisms $C(X^G)\to\Cb^{d_i}$ is given by
\begin{align*}
\rho_{\cP ,\infty} (\varphi , \psi ) &= \max_{f\in\cP} \| \varphi (f) - \psi (f) \|_\infty .
\end{align*}
Take an $\varepsilon > 0$ which is smaller than the minimum over all distinct $Q, Q'\in \cQ$ of the quantities
$\min_{x\in Q, y\in Q'}|x-y|$.
Then it suffices to show
\[\inf_F\inf_{m\in \Nb}\inf_{\delta>0}\limsup_{i\to\infty} \frac{1}{d_i} \log N_{\varepsilon} (\Hom_{\mu^G}^{X^G} (\cP ,F,m,\delta ,\sigma_i ),\rho_{\cP , \infty} )\ge H_{\mu}(\cQ)-\kappa\]
for every $\kappa>0$.

So let $\kappa > 0$.
Let $F$ be a nonempty finite subset of $G$, $m\in \Nb$, and $\delta>0$.
Note that, for each $f\in \cP_{F, m}$, the value of $x\in X^G$ under $f$ depends
only on the coordinates of $x$ at $g$ for $g\in F$. Thus we can find
an $\eta>0$ such that whenever the coordinates of two points $x, y\in X^G$ at $g$ are within $\eta$ of each other
for each $g\in F$, one has $|f(x)-f(y)|<\delta/3$ for all $f\in \cP_{F, m}$.

Take a finite measurable partition $\cR$ of $X$ finer than $\cQ$ such that each atom of $\cR$ ha diameter less than $\eta$.
For each $R\in \cR$ take a point $x_R$ in $R$.
Let $\sigma$ be a map from $G$ to $\Sym(d)$ for some $d\in \Nb$.
For each $\beta\in \cR^{\{1, \dots, d\}}$ take a map $\Theta_{\beta}:\{1, \dots, d\}\rightarrow X^G$ such that for each $a\in \{1, \dots, d\}$ and $g\in F$,
the coordinate of $\Theta_{\beta}(a)$ at $g$ is $x_{\beta(\sigma(g)^{-1}a)}$. Then we have  a unital homomorphism $\Gamma(\beta):C(X^G)\rightarrow \Cb^d$ sending
$f$ to $f\circ \Theta_{\beta}$. Denote by $Z$ the set of $a$ in $\{1, \dots, d\}$ such that $\sigma(e)^{-1}\sigma(g)^{-1}a=\sigma(g)^{-1}a$ for all $g\in F$.
For every $\beta \in \cR^{\{1, \dots, d\}}$, $f\in \cP$ and $a\in Z$, we have
\begin{align*}
(\Gamma(\beta)(\alpha_g(f)))(a)
&=\alpha_g(f)(\Theta_{\beta}(a))=f(g^{-1}\Theta_{\beta}(a))=f((\Theta_{\beta}(a))_g)=f(x_{\beta(\sigma(g)^{-1}a)}),
\end{align*}
and
\begin{align*}
(\sigma(g)\Gamma(\beta)(f))(a)
&=(\Gamma(\beta)(f))(\sigma(g)^{-1}a)=f(\Theta_{\beta}(\sigma(g)^{-1}a))=f((\Theta_{\beta}(\sigma(g)^{-1}a))_e)\\
&=f(x_{\beta(\sigma(e)^{-1}\sigma(g)^{-1}a)})=f(x_{\beta(\sigma(g)^{-1}a)}),
\end{align*}
and hence
\[
(\Gamma(\beta)(\alpha_g(f)))(a)=(\sigma(g)\Gamma(\beta)(f))(a).
\]
When $\sigma$ is a good enough sofic approximation of $G$, one has $1-|Z|/d<\delta^2$, and hence
\[
\|\Gamma(\beta)(\alpha_g(f))-\sigma(g)\Gamma(\beta)(f)\|_2<\delta
\]
for all $\beta \in \cR^{\{1, \dots, d\}}$ and $f\in \cP$.

For each $\varphi\in \cR^F$, denote by $Y_{\varphi}$ the set of $x$ in $X^G$ whose coordinate
at $g$ is in $\varphi(g)$ for every $g\in F$.
Then $\{Y_{\varphi}:\varphi \in \cR^F\}$ is a Borel partition of $X^G$.
For each $\varphi\in \cR^F$ pick a $y_{\varphi}\in Y_{\varphi}$ such that the coordinate of $y_{\varphi}$ at $g$ is $x_{\varphi(g)}$ 
for each $g\in F$.
By our choice of $\eta$ and $\cR$, we have
\[
\sup_{x\in Y_{\varphi}}\sup_{f\in \cP_{F, m}}|f(x)-f(y_{\varphi})|\le \delta/3
\]
for every $\varphi \in \cR^F$.
For every $\beta \in \cR^{\{1, \dots, d\}}$ and $a\in \{1, \dots, d\}$, define $\psi_{\beta, a}\in \cR^F$ by $\psi_{\beta, a}(g)=\beta(\sigma(g)^{-1}a)$. Note that the coordinates of $\Theta_{\beta}(a)$ and $y_{\psi_{\beta, a}}$ at $g$ are the same
for each $g\in F$. For every $\beta\in \cR^{\{1, \dots, d\}}$ and $\varphi\in \cR^F$, one has
\begin{align*}
\big\{a\in \{1, \dots, d\}: \psi_{\beta, a}=\varphi\big\}&=\bigcap_{g\in F}\big\{a\in \{1, \dots, d\}: \psi_{\beta, a}(g)=\varphi(g)\big\}\\
&=\bigcap_{g\in F}\big\{a\in \{1, \dots, d\}: \beta(\sigma(g)^{-1}a)=\varphi(g)\big\}\\
&=\bigcap_{g\in F}\sigma(g)\beta^{-1}(\varphi(g)).
\end{align*}
Thus for every $\beta\in \cR^{\{1, \dots, d\}}$ and $f\in \cP_{F, m}$ one has
\begin{align*}
\zeta(\Gamma(\beta)(f))&=\frac{1}{d}\sum_{a=1}^d\Gamma(\beta)(f)(a)=\frac{1}{d}\sum_{a=1}^df(\Theta_{\beta}(a))=\frac{1}{d}\sum_{a=1}^df(y_{\psi_{\beta, a}})\\
&= \sum_{\varphi\in \cR^F}f(y_{\varphi})\zeta\big(\big\{a\in \{1, \dots, d\}: \psi_{\beta, a}=\varphi\big\}\big)\\
&= \sum_{\varphi\in \cR^F}f(y_{\varphi})\zeta\bigg(\bigcap_{g\in F}\sigma(g)\beta^{-1}(\varphi(g))\bigg).
\end{align*}

Let $\tau$ be a strictly positive number satisfying $\tau<\delta/3$ to be further specified in a moment.
Set
\[
W=\bigg\{\beta\in \cR^{\{1, \dots, d\}}:\sum_{\varphi \in \cR^F}\bigg|\prod_{g\in F}\mu(\varphi(g))
-\zeta\bigg(\bigcap_{g\in F}\sigma(g)\beta^{-1}(\varphi(g))\bigg)\bigg|\le \tau\bigg\}.\]
For every $\beta\in W$ and $f\in \cP_{F, m}$ one has,
since $0\le f\le 1$,
\begin{align*}
\mu^G(f)&=\sum_{\varphi\in \cR^F}\int_{Y_{\varphi}} f\, d\mu^G\approx_{\delta/3}\sum_{\varphi\in \cR^F}\mu^G(Y_{\varphi}) f(y_{\varphi})=\sum_{\varphi\in \cR^F}f(y_{\varphi})\prod_{g\in F}\mu(\varphi(g))\\
&\approx_{\delta/3}\sum_{\varphi\in \cR^F}f(y_{\varphi})\zeta\bigg(\bigcap_{g\in F}\sigma(g)\beta^{-1}(\varphi(g))\bigg)=\zeta(\Gamma(\beta)(f)).
\end{align*}
Therefore, when $\sigma$ is a good enough sofic approximation of $G$, the homomorphism $\Gamma(\beta)$ belongs to
$\Hom^{X^G}_{\mu^G}(\cP, F, m, \delta)$ for every $\beta \in W$.

Let $\beta\in W$, and let us estimate the number of
$\gamma \in W$ satisfying $\rho_{\cP, \infty}(\Gamma(\beta), \Gamma(\gamma))<\varepsilon$.
Note that
\begin{align*}
\rho_{\cP, \infty}(\Gamma(\beta), \Gamma(\gamma))&=\max_{a\in \{1, \dots, d\}}|p(\Theta_{\beta}(a))-p(\Theta_{\gamma}(a))|\\
&=\max_{a\in \{1, \dots, d\}}|x_{\beta(\sigma(e)^{-1}a)}-x_{\gamma(\sigma(e)^{-1}a)}|\\
&=\max_{a\in \{1, \dots, d\}}|x_{\beta(a)}-x_{\gamma(a)}|.
\end{align*}
Thus $\beta(a)$ and $\gamma(a)$ must be contained in the same atom of $\cQ$ for each $a\in \{1, \dots, d\}$.

For $Q\in \cQ$ denote by $\cR_Q$ the set of atoms of $\cR$ contained in $Q$. 
Thinking of $\beta$ as a partition of $\{1, \dots, d\}$ indexed by $\cR$, we see that $\{\bigcap_{g\in F}\sigma(g)\beta^{-1}(\varphi(g)): \varphi\in \cR^F\}$ is a partition of $\{1, \dots, d\}$. 
Let $Q\in \cQ$, $R\in \cR$, and $g_1\in F$. Then
\begin{align*}
\sigma(g_1)\beta^{-1}(R)=\bigcup_{\varphi\in \cR^F, \hspace*{0.3mm}\varphi(g_1)=R}\hspace*{2mm}
\bigcap_{g\in F}\sigma(g)\beta^{-1}(\varphi(g)),
\end{align*}
and hence
\begin{align*}
\zeta(\beta^{-1}(R))=\zeta(\sigma(g_1)\beta^{-1}(R))\approx_{\tau}
\sum_{\varphi\in \cR^F, \hspace*{0.3mm} \varphi(g_1)=R}\hspace*{2mm}\prod_{g\in F}\mu(\varphi(g))=\mu(R),
\end{align*}
and
\begin{align*}
\zeta(\beta^{-1}(\cR_Q))=\zeta\bigg(\bigcup_{R\in \cR_Q}\sigma(g_1)\beta^{-1}(R)\bigg)
&\approx_{\tau} \sum_{\varphi\in \cR^F , \hspace*{0.3mm}\varphi(g_1)\in \cR_Q}\hspace*{2mm}\prod_{g\in F}\mu(\varphi(g)) \\
&=\mu(\cup \cR_Q)=\mu(Q).
\end{align*}
Similarly, we have $\zeta(\gamma^{-1}(R))\approx_{\tau}\mu(R)$ for every $R\in \cR$. The conclusion in the
last paragraph can be restated as saying that $\beta^{-1}(\cR_Q)=\gamma^{-1}(\cR_Q)$ for every $Q\in \cQ$.
Hence the number of possibilities for $\gamma$ is bounded above by
$\prod_{Q\in \cQ}M_{\beta, Q}$, where listing the atoms of $\cR_Q$ as $R_1, \dots, R_n$ we have
\begin{align*}
M_{\beta, Q}&:=\sum_{j_1, \dots, j_n}\binom{|\beta^{-1}(\cR_Q)|}{j_1}\binom{|\beta^{-1}(\cR_Q)|-j_1}{j_2}
\dots \binom{|\beta^{-1}(\cR_Q)|-\sum_{k=1}^{n-1}j_k}{j_n}\\
&=\sum_{j_1, \dots, j_n}\frac{|\beta^{-1}(\cR_Q)|!}{j_1!j_2!\cdots j_n!},
\end{align*}
where the sum ranges over all nonnegative integers $j_1, \dots, j_n$ such that
$|j_k/d-\mu(R_k)|\le \tau$ for all $1\le k\le n$ and $\sum_{k=1}^nj_k=|\beta^{-1}(\cR_Q)|$.
Setting $\xi(t)=-t\log t$ for $0\le t\le 1$, for such $j_1, \dots, j_n$ we have, by Stirling's approximation,
\begin{align*}
\frac{|\beta^{-1}(\cR_Q)|!}{j_1!j_2!\cdots j_n!}
\le Cd\exp\bigg(\bigg(\sum_{k=1}^n\xi(j_k/d)-\xi(|\beta^{-1}(\cR_Q)|/d)\bigg)d\bigg)
\end{align*}
for some constant $C>0$ independent of $|\beta^{-1}(\cR_Q)|$ and $j_1, \dots, j_n$.
Since the function $\xi$ is uniformly continuous,
when $\tau$ is small enough one has
\begin{align*}
\sum_{k=1}^n\xi(j_k/d)-\xi(|\beta^{-1}(\cR_Q)|/d)<\sum_{R\in \cR_Q}\xi(\mu(R))-\xi(\mu(Q))+\kappa/|\cQ|
\end{align*}
for all $j_1, \dots, j_n$ as above.
 Therefore
\begin{align*}
M_{\beta, Q}&\le Cd(2\tau d)^{|\cR_Q|}\exp\bigg(\bigg(\sum_{R\in \cR_Q}\xi(\mu(R))-\xi(\mu(Q))+\kappa/|\cQ|\bigg)d\bigg)
\end{align*}
for every $Q\in \cQ$, and hence
\begin{align*}
\prod_{Q\in \cQ}M_{\beta, Q}\le C^{|\cQ|}d^{|\cQ|}(2\tau d)^{|\cR|}\exp\bigg(\bigg(H_{\mu}(\cR)-H_{\mu}(\cQ)
+\kappa\bigg)d\bigg).
\end{align*}

Now we have
\begin{align*}
\lefteqn{N_{\varepsilon}(\Hom_{\mu^G}^{X^G} (\cP ,F,m,\delta ,\sigma),\rho_{\cP , \infty} )}\hspace*{20mm} \\
\hspace*{20mm} &\ge |W|/\max_{\beta \in W}\prod_{Q\in \cQ}M_{\beta, Q}\\
&\ge |W|C^{-|\cQ|}d^{-|\cQ|}(2\tau d)^{-|\cR|}\exp\bigg(\bigg(-H_{\mu}(\cR)+H_{\mu}(\cQ)-\kappa\bigg)d\bigg).
\end{align*}
Using Lemma~\ref{L-finite lower bound} we thus obtain
\begin{align*}
\lefteqn{\limsup_{i\to\infty} \frac{1}{d_i} \log
N_{\varepsilon} (\Hom_{\mu^G}^{X^G} (\cP ,F,m,\delta ,\sigma_i ),\rho_{\cP , \infty} )}\hspace*{15mm} \\
\hspace*{15mm} &\ge -H_{\mu}(\cR)+H_{\mu}(\cQ)-\kappa
+\limsup_{i\to\infty} \frac{1}{d_i} \log \bigg|\bigg\{\beta\in \cR^{\{1, \dots, d_i\}}: \\
&\hspace*{30mm} \sum_{\varphi \in \cR^F}\bigg|\prod_{g\in F}\mu(\varphi(g))
-\zeta\bigg(\bigcap_{g\in F}\sigma_i(g)\beta^{-1}(\varphi(g))\bigg)\bigg|\le \tau\bigg\}\bigg|\\
&\ge H_{\mu}(\cQ)-\kappa,
\end{align*}
as desired.
\end{proof}

For a standard probability space $(X, \mu)$, the entropy $H(\mu)$ is defined as the supremum of $H_{\mu}(\cQ)$
over all finite measurable partitions $\cQ$ of $X$.
In the case $H(\mu)<+\infty$, the following theorem is Theorem~8.1 of \cite{Bow10} 
in conjunction with Theorem~3.6 of \cite{KerLi10T}. The case $H(\mu)=+\infty$ is a consequence of 
Lemma~\ref{L-lower bound}. When $G$ is amenable we recover a standard computation of classical measure entropy,
in view of \cite{KerLi10A}.

\begin{theorem} \label{T-entropy of Bernoulli}
Let $G$ be a countable sofic group. Let $\Sigma$ be a sofic approximation sequence for $G$.
Let $(X, \mu)$ be a standard probability space. Then $h_{\Sigma, \mu^G}(X^G, G)=H(\mu)$.
\end{theorem}

As a consequence of Theorem~\ref{T-entropy of Bernoulli} we obtain
the following result, which was proved by Bowen in the case that
$H(\mu)+H(\nu)<+\infty$ or $G$ is a countable sofic Ornstein group \cite[Thm.\ 1.1 and Cor.\ 1.2]{Bow10}.
We note that it is not known whether there are countably infinite groups that are not Ornstein.

\begin{theorem} \label{T-isomorphism of Bernoulli}
Let $G$ be a countable sofic group.
Let $(X, \mu)$ and $(Y, \nu)$ be standard probability spaces. If $(G, X^G, \mu^G)$ and $(G, Y^G, \nu^G)$ are isomorphic, 
then $H(\mu)=H(\nu)$.
\end{theorem}

The next lemma follows from Proposition~5.3 of \cite{Bow10} (taking $\beta$ there to be the trivial partition)
and Theorem~3.6 of \cite{KerLi10T}.

\begin{lemma} \label{L-upper bound}
Let $G$ be a countable sofic group. Let $\Sigma$ be a sofic approximation sequence for $G$.
Let $G$ act on a standard probability space $(X,\mu)$ by measure-preserving transformations.
Let $\cQ$ be a generating countable measurable partition of $X$. Then $h_{\Sigma, \mu}(X, G)\le H_{\mu}(\cQ)$.
\end{lemma}

The following theorem is a consequence of Theorem~\ref{T-entropy of Bernoulli} and Lemma~\ref{L-upper bound}.
In the case that $G$ is amenable it is a well-known consequence of classical entropy theory,
and in the case that $G$ contains the free group $F_2$ it was proved by
Bowen in \cite[Thm.\ 1.4]{Bow10}. As mentioned in the introduction, there exist
countable sofic groups that lie outside of these two classes \cite{Ers08,Osi09}.

\begin{theorem} \label{T-countable generator}
Let $G$ be a countable sofic group.
Let $(X, \mu)$ be a standard probability space with $H(\mu)=+\infty$. Then there is no generating countable measurable partition 
$\cQ$ of $X^G$ such that $H_{\mu^G}(\cQ)<+\infty$.
\end{theorem}

\end{document}